\documentclass[letterpaper, 10pt, conference]{ieeeconf} 

\IEEEoverridecommandlockouts         
\usepackage{graphics} 
\usepackage{epsfig} 
\usepackage{mathptmx} 
\usepackage{times} 
\usepackage{amsmath} 
\usepackage{amssymb}  
\usepackage{cite}
\usepackage{bbm}
\usepackage{chemfig,siunitx}
\usepackage{textcomp}
\usepackage{subfigure}

\newtheorem{lemma}{Lemma}
\newtheorem{rem}{Remark}

\title{\LARGE \bf
Controlling Linear Networks with Minimally Novel Inputs}

\author{Gautam Kumar, Delsin Menolascino, MohammadMehdi Kafashan and ShiNung Ching$^{1}$ 
\thanks{Gautam Kumar is with the Department of Electrical and Systems Engineering, Washington University in Saint Louis, Saint Louis, MO 63130, U.S.A.,
        {\tt\small gak206@gmail.com}}
\thanks{Delsin Menolascino is with the Department of Electrical and Systems Engineering, Washington University in Saint Louis, Saint Louis, MO 63130, U.S.A.,
        {\tt\small delsin@ese.wustl.edu}}
\thanks{MohammadMehdi Kafashan is with the Department of Electrical and Systems Engineering, Washington University in Saint Louis, Saint Louis, MO 63130, U.S.A.,
        {\tt\small kafashan@ese.wustl.edu}}
\thanks{$^{1}${\it Corresponding author:} ShiNung Ching is with the Department of Electrical and Systems Engineering $\&$ Division of Biology and Biomedical Sciences, Washington University in Saint Louis, Saint Louis, MO 63130, U.S.A.
        {\tt\small shinung@ese.wustl.edu}}
}

\begin{document}

\maketitle
\thispagestyle{empty}
\pagestyle{empty}

\begin{abstract}
In this paper, we propose a novelty-based metric for quantitative characterization of the controllability of complex networks. This inherently bounded metric describes the average angular separation of an input with respect to the past input history.  We use this metric to find the minimally novel input that drives a linear network to a desired state using unit average energy.  Specifically, the minimally novel input is defined as the solution of a continuous time, non-convex optimal control problem based on the introduced metric.  We provide conditions for existence and uniqueness, and an explicit, closed-form expression for the solution.  We support our theoretical results by characterizing the minimally novel inputs for an example of a recurrent neuronal network.     
  
\end{abstract}

\section{Introduction}

\noindent In its most basic form, the systems-theoretic notion of controllability carries a binary definition: a dynamical system either is, or is not, controllable, with respect to its exogenous inputs. Naturally, such a notion has the deficiency of not grading the ease or difficulty associated with effecting such control. To obviate this issue, consistent research effort has been directed at the characterization of controllability using systems-theoretic metrics. Roughly, these metrics can be grouped into two categories

\begin{enumerate}
	\item Those that characterize the minimum energy parametric perturbations that result in a loss of controllability \cite{Hu2001,HD04}.  These are related to basic characterizations of the robustness of linear systems \cite{hinrichsen1989real}.
	\item Those that characterize the controllability of a system in terms of the minimum energy excitation required to achieve a unit length state trajectory \cite{YRCHL12,PZB14,FZ14,CSL14}.  
\end{enumerate}

     
The latter, in particular, is a natural paradigm that is directly relatable to the celebrated Kalman rank condition (or the controllability gramian) used to ascertain the controllability of linear systems \cite{KS72}. Recently, energy-based controllability metrics have been successfully used in the emerging domain of network science to assess the putative controllability of large-scale linear systems, formulated as complex networks of various topologies \cite{FZ14,PZB14}. However, for complex networks in general and, in particular, for biological neuronal networks, an energy-based metric offers insight into only one aspect of the overall system's controllability.

\begin{figure}
	\centering
		\includegraphics[width=0.40\textwidth]{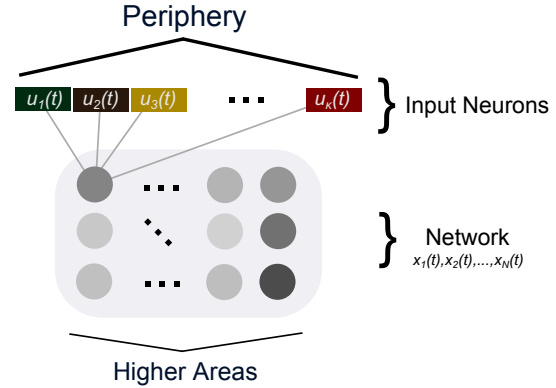}
	\caption{Prototypical structure of a sensory neuronal network. Sensory neurons are tuned to features from the sensory periphery. These neurons project excitation onto a network that performs intermediate transformations on the afferent excitation en route to higher brain regions.}
	\label{fig1}
\end{figure}

We appeal, specifically, to the domain of neural coding and the dynamics of sensory neural circuits. Consider the simple, prototypical layered model of a sensory network shown in Figure \ref{fig1}, wherein sensory neurons are tuned to a high dimensional feature space (i.e., environmental variables from the sensory periphery; say, different molecules corresponding to tastes). Those sensory neurons impinge on a complex, interconnected sensory network that performs intermediate transformations en route to higher brain areas.

One may put forth a supposition that the `controllability' of such a sensory network, with respect to the afferent input from the sensory neurons, is critical in mediating the ability to perceive minute changes in the environment. But as much as energy is important is mediating such a response, \textit{orientation}, i.e., the alignment of an input with certain features, may be even more so. Indeed, a weak, but highly \textit{novel} input may be more easily perceived than an intense, but more familiar, stimulus. The ability to assess the responsiveness of neuronal networks to novelty -- at a particular moment in time, relative to past inputs -- has immediate implications in the analysis and control of biophysiological neuronal network dynamics in different behavioral and clinical regimes \cite{Ching2013,Ching2012,Lepage2013a}.

Here, as a first step, we seek to characterize the controllability of linear systems (linear networks) possessing high dimensional input-spaces, with respect to input novelty. In particular, we ask how responsive are the state (node) trajectories to inputs that differ in orientation from those that have previously been applied. Figure \ref{fig2} illustrates the basic notion of input novelty for a simple two-dimensional linear system with three inputs. A particular input drives the system to an intermediate point in the phase space; from this point emerge two trajectories, both of which reach a common endpoint; one minimizes input novelty (note the similarity between the input from $t\in [0,2]$ and that from $t\in [2,4]$), the other minimizes energy. 

\begin{figure}[t]
\centering%
\includegraphics[width=0.48\textwidth]{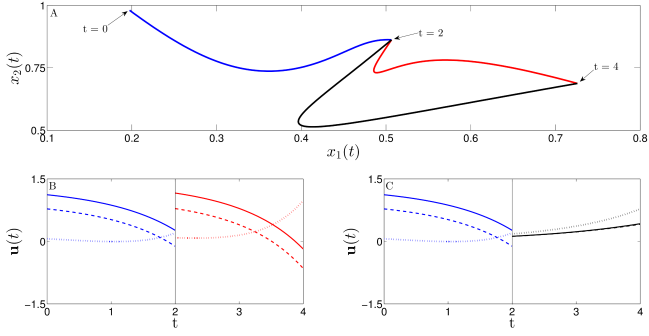}
\caption{Minimum novelty control vs. minimum energy control: (A) The trajectory (blue) brings the system from an initial state on intermediate state at $t=2s$.  Subsequently, two trajectories are contrasted in the phase-plane for the minimum novelty control (red) and the minimum energy control (black). (B) The minimally novel inputs (from $t=2s$ to $t=4s$) (red) designed using our approach in this paper. (C) The inputs corresponding to the minimum energy trajectory (from $t=2s$ to $t=4s$), (black).}
\label{fig2}
\end{figure}  


Specifically, we: (i) analytically derive the \textit{minimum novelty control} for linear networks by formulating a non-convex optimization problem. The problem seeks the minimum angular separation, defined in terms of an inner product in the input feature space, required in order to create a desired change in the network trajectory, constrained by a fixed average input energy; and (ii) characterize the resulting cost -- the control `novelty' -- that describes the change in input orientation that is required to drive the system to a given state.

The remaining paper is as follows. In section \ref{sec2}, we introduce our inner-product based controllability metric for linear networks and formulate a non-convex optimal control problem that minimizes this metric under the constraint of unit average energy. In section \ref{sec3}, we establish the existence and the uniqueness of a global optimal solution of the control problem and derive a closed-form expression for minimally novel inputs. Finally, in section \ref{sec4}, we consider a linearized firing rate model of a recurrent neuronal network as an example to demonstrate our theoretical results.  The paper concludes with a summary and discussion of future work.

\section{Problem Formulation}\label{sec2}
\subsection{Mathematical notation}
\noindent Most notation is standard and will be introduced as the results are developed. We use lower-case letters to represent scalars, boldface lower-case letters to represent vectors, capital letters to represent matrices. Exceptions are $T$, $\mathbb{J}(T)$ and $\mathbb{J}_{1}(T)$, which we represent as scalars. We use $\mathbb{R}^{n\times 1}$ to represent the space of $n$- dimensional vectors with their elements as real numbers. Similarly, we use $\mathbb{R}^{n\times m}$ and $\mathbb{R}_{+}^{n\times m}$ to represent the space of $n\times m$ dimensional matrices with their elements as real numbers and non-negative real numbers respectively. $\|\mathbf{x}\|_{2}$ is the Euclidean norm of the vector $\mathbf{x}$. $\mathbf{x}^{'}$ is the transpose of a vector $\mathbf{x}$ and $A^{-1}$ is the inverse of a matrix $A$.   

\subsection{Input novelty based controllability metric}
\noindent We consider a linear, time invariant system with dynamics of the form
\begin{equation}
	\frac{\rm{d}\mathbf{x}(t)}{\rm{d}t} = A \mathbf{x}(t) + B \mathbf{u}(t)
\label{eq1}
\end{equation} 
Here $\mathbf{x}(t)\in \mathbb{R}^{n\times 1}$ represents the state of the system at time $t$, $A\in \mathbb{R}^{n\times n}$ is the state transition matrix, $B\in \mathbb{R}^{n\times m}$ is the input matrix, and $\mathbf{u}(t)\in \mathbb{R}^{m\times 1}$ is the input to the system. Without loss of generality, we say that (\ref{eq1}) describes the time evolution of linear networks in the presence of external inputs.

Let us assume an input $\mathbf{v}(t-T) \in \mathbb{R}^{m\times 1}$, $t\in[0,T]$, with total energy $T$, i.e.
\begin{equation}
\frac{1}{T}\int_{0}^{T}\|\mathbf{v}(t-T)\|_{2}^{2}\rm{d}t = 1 
\label{eq2}
\end{equation}
We assume that $\mathbf{v}(t-T)$ can drive $\mathbf{x}(t)$ from $\mathbf{x}(-T)$ to $\mathbf{x}(0)$, where $\|\mathbf{x}(0)\|_{2} = 1$, subject to the dynamics (\ref{eq1}). Here $T>0$ is a constant. We introduce the inner-product based metric   

\begin{equation}
	\mathbb{J}(T) = \frac{1}{T}\int_{0}^{T} \mathbf{v}^{'}(t-T)\mathbf{u}(t)\rm{d}t
\label{eq3}
\end{equation}

\noindent where 
\begin{equation}
\frac{1}{T}\int_{0}^{T}\|\mathbf{u}(t)\|_{2}^{2}\rm{d}t = 1,
\label{eq4}
\end{equation}
\noindent to measure the novelty of a subsequent input $u(t)$, $t\in[0,T]$, relative to $v(t-T)$, required in order to reach the state $x(T)$, where $\|\mathbf{x}(T)\|_{2} = 1$. In other words for a fixed input energy, the metric $\mathbb{J}(T)$ measures the required directional change in inputs (thus novelty) to achieve a given state (or, equivalently, directional) change in the state of the system. 

\begin{rem}\label{rem1}
It is readily evident that $\mathbb{J}(T)\in [-1,1]$. 
\end{rem}


\begin{rem}\label{rem2}
From (\ref{eq3}), we note that the novelty of the input $\mathbf{u}(t)$ compared to $\mathbf{v}(t-T)$ decreases as $\mathbb{J}(T)$ increases and is minimum when $\mathbb{J}(T) = 1$ i.e. when $\mathbf{u}(t) = \mathbf{v}(t-T)$ for all $t\in[0,T]$.
\end{rem}

\begin{rem}\label{rem3}
	 We observe that, due to the energy normalization in (\ref{eq2}) and (\ref{eq4}),
\begin{equation}
\frac{1}{T}\int_{0}^{T} \|\mathbf{v}(t-T)-\mathbf{u}(t)\|_{2}^{2}\rm{d}t = 2 (1- \mathbb{J}(T))
\label{eq5}
\end{equation}
\noindent Thus, the average Euclidean distance, i.e. the left hand side of (\ref{eq5}), between two inputs can equivalently be used as an alternate measure of input novelty in our context.   
\end{rem}

\subsection{Minimum novelty problem}
\noindent From the conceptual formulation introduced above, we can develop a control problem to design the minimally novel input $\mathbf{u}(t), t\in[0,T]$ such that a desired directional change in the state of the system can be achieved under the constraint of fixed energy subject to the system dynamics (\ref{eq1}). For this, we formulate the following optimal control problem:


\begin{subequations}
\begin{align}
    \min_{\substack{\mathbf{u}(t) \\ t\in[0,T]}}
        & \quad  -\mathbb{J}(T) \label{eq6a} \\
    \textrm{s.t.} 
        & \qquad \frac{1}{T} \int_{0}^{T}\| \mathbf{u}(t)\|_{2}^{2}\rm{d}t = 1 \label{eq6b}\\
        & \quad  \mathbf{x}(T) = e^{AT}\mathbf{x}(0) + \int_{0}^{T}e^{A(T-t)}B\mathbf{u}(t)\rm{d}t  \label{eq6c}
\end{align}
\label{eq6}
\end{subequations}

\noindent It should be noted here that the constraint (\ref{eq6c}) is obtained by integrating (\ref{eq1}) with respect to $t$ over the period of $[0,T]$. Immediately, we note that the quadratic equality constraint (\ref{eq6b}) makes the optimization problem (\ref{eq6}) non-convex. Furthermore, we note that our optimal control problem formulation (\ref{eq6}) is different from the classical minimum effort problems where the $L^{1}$-norm of control inputs is minimized under the constraints of explicit lower and upper bounds on the inputs. 


\section{Results}\label{sec3}
\noindent We derive conditions for the existence of a unique global optimal solution of the non-convex optimization problem (\ref{eq6}). Based on this, we provide a closed-form expression for the optimal $\mathbf{u}(t), t\in[0,T]$.

\subsection{Existence of a Minimally Novel Input}
  
\begin{lemma}\label{lem1}
	A solution of the non-convex optimization problem (\ref{eq6}) exists if 
\begin{equation}
	T > \max{\{\mathbf{s}^{'}(T)W_{c}^{-1}(T)\mathbf{s}(T),\mathbf{r}^{'}(T)W_{c}^{-1}(T)\mathbf{r}(T)\}}
\label{eq7}
\end{equation}
\noindent where 
\begin{subequations}
\begin{equation}
\mathbf{s}(T) = \int_{0}^{T}e^{A(T-t)}B\mathbf{v}(t-T)\rm{d}t
\label{eq8a}
\end{equation}
\begin{equation}
\mathbf{r}(T) = \mathbf{x}(T)-e^{AT}\mathbf{x}(0)
\label{eq8b}
\end{equation}
\label{eq8}
\end{subequations}


\noindent Here, $W_{c}(T)$ is the controllability gramian at time $T$ and is defined as 
\begin{equation}
	W_{c}(T) = \int_{0}^{T}e^{A(T-t)}BB^{'}e^{A^{'}(T-t)}\rm{d}t
\label{eq9}
\end{equation}
\end{lemma}
Recall that by our formulation, $T$ is the total energy available to the system (\ref{eq1}). 

\begin{rem}\label{rem4}
The arguments $\mathbf{s}^{'}(T)W_{c}^{-1}(T)\mathbf{s}(T)$ and $\mathbf{r}^{'}(T)W_{c}^{-1}(T)\mathbf{r}(T)$ in (\ref{eq7}) are the minimum energy required to drive the system (\ref{eq1}) from $\mathbf{x}(-T)$ to $\mathbf{x}(0)$ and $\mathbf{x}(0)$ to $\mathbf{x}(T)$ respectively \cite{KS72,CSL14}.
\end{rem}


\begin{proof}
Define $y(t)$ as 
\begin{equation}
	y(t) = \frac{1}{T}\int_{0}^{t}\| \mathbf{u}(\tau)\|_{2}^{2}\rm{d}\tau 
\label{eq10}
\end{equation}
Clearly, $y(0) = 0$ and $y(T) = 1$ from (\ref{eq6b}). Thus, we can replace the constraint (\ref{eq6b}) by
\begin{equation}
	y(T) = 1
\label{eq11}
\end{equation}
\noindent In differential form, we can write (\ref{eq10}) as
\begin{equation}
	\frac{\rm{d}y(t)}{\rm{d}t} = \frac{1}{T}\| \mathbf{u}(t)\|_{2}^{2}
\label{eq12}
\end{equation}

\noindent To solve the dynamic optimization problem (\ref{eq6a}), (\ref{eq6c}) and (\ref{eq11}) in continuous time, we write the Hamiltonian $\mathcal{H}(\mathbf{x}(t),y(t),\mathbf{u}(t),\mathbf{\lambda}(t),\mu(t),t)$ as
\begin{multline}
	\mathcal{H}(\mathbf{x}(t),y(t),\mathbf{u}(t),\mathbf{\lambda}(t),\mu(t),t) = -\frac{1}{T} \mathbf{v}^{'}(t-T)\mathbf{u}(t) \\
	+ \mathbf{\lambda}^{'}(t)(A\mathbf{x}(t)+B\mathbf{u}(t)) \\
	+ \frac{\mu(t)}{T}\| \mathbf{u}(t)\|_{2}^{2}
\label{eq13}
\end{multline}
\noindent Here, $\mathbf{\lambda}(t)$ and $\mu(t)$ are the costate variables associated with the dynamics (\ref{eq1}) and (\ref{eq12}) respectively. We derive the following optimality conditions (i.e. the Euler-Lagrange equations \cite{K04}): 
\begin{subequations}
	\begin{equation}
		\frac{\rm{d}\mathbf{\lambda}(t)}{\rm{d}t} = -(\frac{\partial \mathcal{H}(\mathbf{x}(t),y(t),\mathbf{u}(t),\mathbf{\lambda}(t),\mu(t),t)}{\partial \mathbf{x}(t)})^{'} = -A^{'}\mathbf{\lambda}(t)
	\label{eq14a}
	\end{equation}
	\begin{equation}
		\frac{\rm{d} \mu(t)}{\rm{d}t} = -\frac{\partial  \mathcal{H}(\mathbf{x}(t),y(t),\mathbf{u}(t),\mathbf{\lambda}(t),\mu(t),t)}{\partial y(t)}= 0
	\label{eq14b}
	\end{equation}
	\begin{equation}
		\begin{split}
			\frac{\partial \mathcal{H}(\mathbf{x}(t),y(t),\mathbf{u}(t),\mathbf{\lambda}(t),\mu(t),t)}{\partial \mathbf{u}(t)} &= 0 \\
		&= \frac{2\mu(t)}{T}\mathbf{u}(t) - \frac{1}{T}\mathbf{v}(t-T) \\
		&+ B^{'}\mathbf{\lambda}
	\end{split}
	\label{eq14c}
	\end{equation}
\label{eq14}
\end{subequations}

\noindent By integrating the costate equations (\ref{eq14a}) and (\ref{eq14b}) over $t$, we obtain
\begin{subequations}
	\begin{equation}
		\mathbf{\lambda}(t) = e^{-A^{'}t}\mathbf{\lambda}(0)
	\label{eq15a}
	\end{equation}
	\begin{equation}
		\mu(t) \equiv \mu \qquad \forall t\in[0,T]
	\label{eq15b}
	\end{equation}
\label{eq15}
\end{subequations}

\noindent Here, $\mathbf{\lambda}(0)$ is the initial condition (at $t=0$) of (\ref{eq14a}). From (\ref{eq14c}), (\ref{eq15a}) and (\ref{eq15b}), we derive the optimal control law as
\begin{equation}
	\mathbf{u}(t) = \frac{1}{2\mu}\mathbf{v}(t-T) - \frac{T}{2\mu}B^{'}e^{-A^{'}t}\mathbf{\lambda}(0)
\label{eq16}
\end{equation}

\noindent By substituting (\ref{eq16}) into (\ref{eq6c}), we obtain $\mathbf{\lambda}(0)$ as
\begin{equation}
	\mathbf{\lambda}(0) = e^{A'T}W_{c}^{-1}(T)(\frac{1}{T}\mathbf{s}(T) - \frac{2\mu}{T}\mathbf{r}(t))
\label{eq17}
\end{equation}

\noindent  By substituting (\ref{eq16}) and (\ref{eq17}) in (\ref{eq11}) and using (\ref{eq2}), we obtain
\begin{equation}
	\mu = \pm\frac{1}{2} \sqrt{\frac{T  - \mathbf{s}^{'}(T)W_{c}^{-1}(T)\mathbf{s}(T)}{T -\mathbf{r}^{'}(T)W_{c}^{-1}(T)\mathbf{r}(T)}}
\label{eq18}
\end{equation}

\noindent For the existence of a solution, $\mu$ must be a real number. Thus, either $T < \min{\{\mathbf{s}^{'}(T)W_{c}^{-1}(T)\mathbf{s}(T),\mathbf{r}^{'}(T)W_{c}^{-1}(T)\mathbf{r}(T)\}}$ or $T > \max{\{\mathbf{s}^{'}(T)W_{c}^{-1}(T)\mathbf{s}(T),\mathbf{r}^{'}(T)W_{c}^{-1}(T)\mathbf{r}(T)\}}$. Now it follows directly from Remark \ref{rem4} that the total energy $T$ must satisfy (\ref{eq7}) for the existence of a solution i.e. $T > \max{\{\mathbf{s}^{'}(T)W_{c}^{-1}(T)\mathbf{s}(T),\mathbf{r}^{'}(T)W_{c}^{-1}(T)\mathbf{r}(T)\}}$. 

\end{proof}

\subsection{Uniqueness of the Minimally Novel Input}

\begin{lemma}\label{lem2}
	Under the hypothesis of Lemma \ref{lem1}, the solution of the non-convex optimization problem (\ref{eq6}) is unique.
\end{lemma}

\begin{proof}
\noindent By substituting (\ref{eq16}) and (\ref{eq17}) in (\ref{eq3}), we obtain the optimal value of $\mathbb{J}(T)$ as a function of $\mu$ as
\begin{multline}
	\mathbb{J}(T) = \frac{1}{T}\mathbf{s}^{'}(T)W_{c}^{-1}(T)\mathbf{r}(T)\\
	+\frac{1}{2\mu}(1-\frac{1}{T}\mathbf{s}^{'}(T)W_{c}^{-1}(T)\mathbf{s}(T))
\label{eq19}
\end{multline}

\noindent It follows from Lemma (\ref{lem1}) that $\frac{1}{T}\mathbf{s}^{'}(T)W_{c}^{-1}(T)\mathbf{s}(T) \in (0,1)$ (see (\ref{eq7})). Thus, the maximum of $\mathbb{J}(T)$ occurs when $\mu > 0$ in (\ref{eq18}) i.e.  
\begin{equation}
	\mu = \frac{1}{2} \sqrt{\frac{T  - \mathbf{s}^{'}(T)W_{c}^{-1}(T)\mathbf{s}(T)}{T -\mathbf{r}^{'}(T)W_{c}^{-1}(T)\mathbf{r}(T)}}
\label{eq20}
\end{equation}

\noindent Thus, a unique optimal control input $\mathbf{u}(t)$ exists and is given by
\begin{multline}
	\mathbf{u}(t) = \frac{1}{2\mu}(\mathbf{v}(t-T)-B^{'}e^{A^{'}(T-t)}W_{c}^{-1}(T)\mathbf{s}(T))\\
	+ B^{'}e^{A^{'}(T-t)}W_{c}^{-1}(T)\mathbf{r}(T)
\label{eq21}
\end{multline}
\end{proof}

\subsection{Euclidean - Inner Product Equivalence}
\noindent As noted in Remark \ref{rem3}, it is an interesting and notable consequence of our cost formulation that the problem can exactly recast in terms of a Euclidean norm. Specifically, if we consider 

\begin{equation}
\mathbb{J}_{1}(T) = \frac{1}{T}\int_{0}^{T} \|\mathbf{v}(t-T)-\mathbf{u}(t)\|_{2}^{2}\rm{d}t
\label{eq22}
\end{equation}

\noindent  as the cost function in (\ref{eq6a}), we obtain the optimal solution as
\begin{subequations}
	\begin{equation}
		\mu = -1 + \sqrt{\frac{T  - \mathbf{s}^{'}(T)W_{c}^{-1}(T)\mathbf{s}(T)}{T -\mathbf{r}^{'}(T)W_{c}^{-1}(T)\mathbf{r}(T)}}
	\label{eq23a}
	\end{equation}
	\begin{multline}
		\mathbf{u}(t) = \frac{1}{1+\mu}(\mathbf{v}(t-T)-B^{'}e^{A^{'}(T-t)}W_{c}^{-1}(T)\mathbf{s}(T))\\
	+ B^{'}e^{A^{'}(T-t)}W_{c}^{-1}(T)\mathbf{r}(T)
	\label{eq23b}
	\end{multline}
	\begin{multline}
		\mathbb{J}_{1}(T) = 2(1-\frac{1}{T}\mathbf{s}^{'}(T)W_{c}^{-1}(T)\mathbf{r}(T))\\
		+\frac{2}{1+\mu}(\frac{1}{T}\mathbf{s}^{'}(T)W_{c}^{-1}(T)\mathbf{s}(T)-1)
	\label{eq23c}
	\end{multline}
\label{eq23}
\end{subequations} 

\noindent It is evident that the control law (\ref{eq20})-(\ref{eq21}) is same as the control law (\ref{eq23a})-(\ref{eq23b}), as one expects from Remark \ref{rem3}.

\section{Example}\label{sec4}
\noindent \noindent We consider a recurrent network of $n$ neurons with linearized firing rate dynamics of the form \cite{DA01} 
\begin{equation}
	S\frac{\rm{d}\mathbf{x}(t)}{\rm{d}t} = -\mathbf{x}(t) + W \mathbf{x}(t) + B\mathbf{u}(t)
\label{eq24}
\end{equation}

\noindent Here, $\mathbf{x}(t) \in \mathbb{R}_{+}^{n\times 1}$ represents the firing rate of the neurons at time $t$, $S \in \mathbb{R}_{+}^{n\times n}$ is a diagonal matrix whose diagonal elements are the (positive) time constants of the neurons, $W \in \mathbb{R}^{n\times n}$ defines the interaction among neurons in the network (weight matrix), $B \in \mathbb{R}_{+}^{n\times n}$ is the input matrix, and $\mathbf{u}(t)\in \mathbb{R}_{+}^{n\times 1}$ is the afferent input. Since $S$ is invertible, (\ref{eq24}) can be represented in the form of (\ref{eq1}) by considering $A = S^{-1}(-I+W)$ where $I$ is the $n\times n$ identity matrix. 

For illustrative purposes, we consider a recurrent network of $n = 100$ neurons where $80$ neurons are excitatory and every $5^{th}$ neuron is inhibitory. We choose the time constants (in ms) of the neurons, i.e. the elements of the diagonal matrix $S$, from a uniform distribution $\mathcal{U}(5,10)$. For every excitatory neuron $i$, we choose the connectivity weight $w_{i,j}$ (in essence, a time constant for excitation from the neuron $i$ to $j$) from a uniform distribution $\mathcal{U}(0,1)$. Similarly, for every inhibitory neuron $i$, we choose the connectivity weight $w_{i,j}$ (from the neuron $i$ to $j$) from a uniform distribution $\mathcal{U}(-1,0)$. We assume that $w_{i,j} = 0$ for $i=j$, i.e. neurons do not possess direct feedback. Assuming $B$ as an identity matrix, we proceed to compute the minimum directional change in inputs (i.e. minimally novel inputs) required to make a desired directional change in firing rates of neurons using (\ref{eq8})-(\ref{eq9}), (\ref{eq19})-(\ref{eq21}).

To complete the example, we specify $T = 3$ ms. The initial and terminal states $\mathbf{x}(0)$ and $\mathbf{x}(T)$, respectively, are specified to satisfy $\|\mathbf{x}(0)\|_{2} = \|\mathbf{x}(T)\|_{2} = 1$ with $\mathbf{x}(0)'\mathbf{x}(T) = \gamma$, where in this particular case we specify $\gamma = 0.7645$. The prior input $\mathbf{v}(t-T)$ is specified to be constant over the interval $t\in[0,T]$. Figure \ref{fig3} illustrates the outcome of the example for $n=1000$ random realizations of the system. Each red dot on the figure depicts the novelty associated with the solution to (\ref{eq18})-(\ref{eq19}), i.e., the minimum novelty. Note, again, that by formulation, these inputs all have unit average energy. Each blue dot corresponds to the minimum energy solution. As a verification of our theoretical development, we note that the minimum energy solution consistently requires an injection of novelty (angular orientation) relative to the prior input and relative to the optimum.

\begin{figure}[ht]
\centering%
\includegraphics[width=0.48\textwidth]{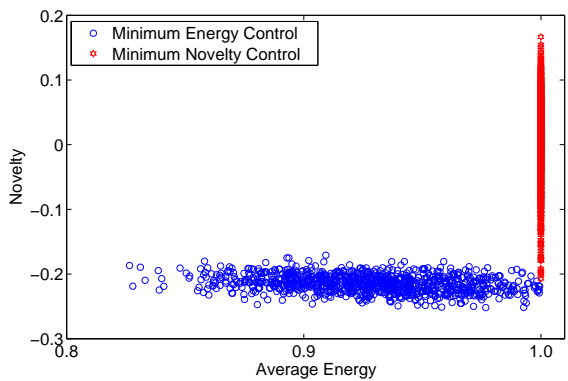}
\caption{Comparison of minimum novelty control with minimum energy control for $n=1000$ random realizations of the recurrent neuronal network: Each red dot on the figure depicts the novelty associated with the solution to the minimum novelty control. Each blue dot corresponds to the minimum energy solution.}
\label{fig3}
\end{figure}

\section{Conclusions and Future Work}
\noindent In this paper, we have introduced a systems-theoretic analysis to characterize the minimum input novelty required to effect a change to the trajectory of a linear system. We have focused this paper on introducing the key conceptual notion and on exact analytical characterization of the minimum novelty solution for the case of linear systems. Naturally, several extensions are possible and some are immediate. For instance, the analysis readily extends to the case of linear-time varying systems, with appropriate replacement of the static $A$ and $B$ matrices with their time-varying equivalents in the controllability gramian. Depending on the domain example at hand, one may also modify the novelty metric itself, for instance by weighting novelty in certain segments of the state-space.

The compelling aspect of this analysis is its direct interpretability in the context of sensory neuronal networks where, as stated in the Introduction, energy alone does not provide a full controllability characterization. With suitable adaptation, it is expected that the analysis herein can be used for both the analysis of biophysical neuronal networks in clinically relevant regimes \cite{Ching2014} and, eventually, for design and synthesis of sensory inputs \cite{KSTK13}.

\section{Acknowledgement} 
\noindent S. Ching Holds a Career Award at the Scientific Interface from the Burroughs-Wellcome Fund.

\IEEEtriggeratref{2}
\bibliographystyle{IEEEtran}
\bibliography{masterN}


\end{document}